\documentclass[a4paper]{amsart} 
\usepackage{graphicx} 

\usepackage{amsmath}
\usepackage{amssymb}
\usepackage{amsfonts}
\usepackage{amsthm}
\usepackage{latexsym}

\usepackage{comment} 

\newtheorem{definition}{Definition}[section]
\newtheorem{theorem}[definition]{Theorem}
\newtheorem{proposition}[definition]{Proposition}
\newtheorem{lemma}[definition]{Lemma}
\newtheorem{corollary}[definition]{Corollary}
\newtheorem{claim}[definition]{Claim}

\newcommand{\deflogic}{\mathcal{L}}
\newcommand{\defframe}{\mathcal{K}}
\newcommand{\defalg}{\mathcal{A}}

\newcommand{\mdl}{\Box}

\newcommand{\logick}{\mathbf{K}}

\newcommand{\propgl}{\mathbf{GL}}

\newcommand{\pglsys}{\mathsf{NGL}}

\newcommand{\framefi}{\mathfrak{FI}}
\newcommand{\framecw}{\mathfrak{CW}}

\newcommand{\framelf}{\mathfrak{LF}}

\newcommand{\iand}{\bigwedge}

\newcommand{\da}{\downarrow}

\newcommand{\thn}{\ \Rightarrow\ }
\newcommand{\eq}{\ \Leftrightarrow\ }

\newcommand{\yy}{\rightarrow}

\newcommand{\gldb}{\db_{\mathsf{NGL}}}

\newcommand{\db}{\vdash}
\newcommand{\vl}{\models}

\newcommand{\gm}{\Gamma}

\newcommand{\power}{\mathcal{P}}

\newcommand{\propvar}{\mathsf{Prop}}

\newcommand{\qflt}{\mathcal{F}_{Q}}

\newcommand{\nec}{\Box}

\newcommand{\dframe}{\mathsf{Frm}_{Q}}
\newcommand{\dalg}{\mathsf{Alg}}

\newcommand{\dmndn}{\text{\rm ($\Diamond^{\ast}$)}}

\newcommand{\frmwk}{\mathfrak{F}_{\text{{\rm W}$\Diamond^{\ast}$}}}
\newcommand{\frmst}{\mathfrak{F}_{\text{{\rm S}$\Diamond^{\ast}$}}}

\begin{document}

\title[An $\omega$-rule for the logic of provability]
{An $\omega$-rule for the logic of provability and its models}
\author{Katsumi Sasaki}
\author{Yoshihito Tanaka}
\date{}

\maketitle

\begin{abstract}
In this paper, we discuss a proof system $\pglsys$
for the logic $\propgl$ of provability,  
which is equipped with an $\omega$-rule.
We show the three classes of transitive Kripke frames, 
the class which strongly validates the $\omega$-rule, 
the class which weakly validates the $\omega$-rule,  
and 
the class which is defined by the L\"{o}b formula, 
are mutually different, while all of them characterize $\propgl$. 
This gives an
example of a proof system $P$ and 
a class $C$ of Kripke frames such that $P$ is sound with respect to $C$ but 
the soundness cannot be proved by simple induction on the height of the derivations 
in $P$.  
We also show Kripke completeness of $\pglsys$ in an algebraic manner. 
As a corollary, we show that the class of modal algebras which is defined by 
equations 
$\mdl x\leq\mdl\mdl x$ and $\iand_{n\in\omega}\Diamond^{n}1=0$ 
is not a variety. 
\end{abstract}

\section{Introduction}

In this paper, we discuss a proof system $\pglsys$ for the logic $\propgl$ of provability,   
which is equipped with an $\omega$-rule.  
We show the three classes of transitive Kripke frames, 
the class which strongly validates the $\omega$-rule, 
the class which weakly validates the $\omega$-rule,  
and 
the class which is defined by the L\"{o}b formula, 
are mutually different, while all of them characterize $\propgl$. 
This gives an
example of a proof system $P$ and 
a class $C$ of Kripke frames such that $P$ is sound with respect to $C$ but 
the soundness cannot be proved by simple induction on the height of the derivations 
in $P$.  
We also show Kripke completeness of $\pglsys$ in an algebraic manner. 
As a corollary, we show that the class of modal algebras which is defined by 
equations 
$\mdl x\leq\mdl\mdl x$ and $\iand_{n\in\omega}\Diamond^{n}1=0$ 
is not a variety.

It is known that $\propgl$ is sound and complete 
with respect to the class $\framecw$ of transitive and conversely well-founded frames
and also to the class $\framefi$ of finite, transitive, and irreflexive 
Kripke frames
(e.g., \cite{bls93,hgh-crs96,chg-zkh97,blc-rjk-vnm01}).  
Therefore, $\propgl$ is sound and complete with respect to any class $C$ of 
Kripke frames such that  $\framefi\subseteq C\subseteq\framecw$.

One of such classes is the class $\framelf$ of 
transitive Kripke frames of locally finite height
(see Definition \ref{defframelf}). 
In \cite{tnk18}, a cut-free proof system with an $\omega$-rule for a predicate 
extension of $\propgl$ is introduced, and completeness of the system with respect to 
$\framelf$ is proved. 
The proof system in \cite{tnk18} is defined in Gentzen-style, but 
the $\omega$-rule in it is 
essentially same 
as the following: 
$$
\dmndn: \ 
\frac
{\phi\supset\Diamond^{n}\top\hspace{10pt}
(\text{$\forall n\in\omega$})}
{\phi\supset\bot}.
$$

In this paper, we introduce a proof system $\pglsys$ for $\propgl$ which 
is equipped with the $\omega$-rule $\dmndn$, and 
discuss two classes $\frmwk$ and $\frmst$ of transitive Kripke frames in which 
the rule 
$\dmndn$ is weakly valid and strongly valid, respectively
(see Definition \ref{defrulevl}). 
We show the following relations hold among four classes of transitive Kripke frames: 
\begin{equation}\label{mainresult}
\framelf
=
\frmst
\subsetneqq
\frmwk
\subsetneqq
\framecw.  
\end{equation}
By $\frmwk\subsetneqq\framecw$, 
it follows that 
the pair $\pglsys$ and $\framecw$ is an 
example of a proof system $P$ and 
a class $C$ of Kripke frames such that $P$ is sound and complete with respect to $C$ but 
the soundness cannot be proved by simple induction on the height of the derivations 
in $P$.

While the Kripke completeness of $\pglsys$ with 
respect to $\framelf$ is proved in \cite{tnk18} by 
Henkin-construction, 
we give another proof of it by means of modal algebras.  
It is known that 
Kripke completeness of many kinds of modal logics 
follows from the J\'{o}nsson-Tarski representation of modal algebras
\cite{jns-trs51,jns-trs52,chg-zkh97,blc-rjk-vnm01}. 
However, it is not enough 
to prove Kripke completeness of logics such as predicate modal logics, infinitary modal logics, 
or modal logics with $\omega$-rules, 
as the embedding given in it
does not preserve infinite meets nor joins, in general. 
To deal with such logics, 
an infinitary extension of 
the J\'{o}nsson-Tarski representation 
is introduced in \cite{tnk-ono98}, and is used 
to show Kripke completeness of 
predicate modal logics \cite{tnk-ono98}, infinitary modal logics \cite{tnk-ono98},  and
modal logics with $\omega$-rules \cite{tnknoncpt}.  
In this paper, we introduce another infinitary extension of 
the J\'{o}nsson-Tarski representation 
for the modal algebras which satisfy $\iand_{n\in\omega}\Diamond^{n}1=0$. 
This representation theorem 
can be applied to some modal algebras which do not satisfy the conditions 
of the infinitary representation theorem in 
\cite{tnk-ono98}. 
As a corollary, we show that the class of modal algebras which is defined by equations
$\mdl x\leq\mdl\mdl x$ and $\iand_{n\in\omega}\Diamond^{n}1=0$ 
is not a variety.

The construction of this paper is the following. 
In Section \ref{syntax_semantics}, we fix definitions and notations and 
recall basic properties of modal logic. 
In Section \ref{modal_algebras}, 
we introduce the infinitary extension of 
the J\'{o}nsson-Tarski representation.
In Section \ref{non-compact_rule}, 
we introduce the system $\pglsys$ and 
show its Kripke completeness. 
In Section \ref{modelsofpglsys}, we discuss classes of Kripke frames which 
characterize $\propgl$.

\section{Preliminaries}\label{syntax_semantics}

In this section, we recall basic definitions and properties of modal logic. 
The language we consider consists of the 
following symbols:
\begin{enumerate} 
\item
a countable set $\propvar$ of propositional variables;

\item
$\top$ and $\bot$;

\item logical connectives: 
$\land$,  $\neg$;

\item
modal operator $\mdl$.
\end{enumerate}

The set $\Phi$ of formulas is defined recursively as follows:

\begin{enumerate}
\item
$\top$, $\bot$, 
and each $p\in\propvar$ are in $\Phi$;

\item 
if $\phi$ and $\psi$ are in $\Phi$
then $(\phi\land\psi)\in \Phi$;

\item 
if $\phi\in\Phi$ 
then $(\neg\phi)\in\Phi$, and  
$(\mdl\phi)\in\Phi$.
\end{enumerate}

The symbols $\lor$ and $\supset$ are defined in a usual way.  
We write 
$\phi\equiv\psi$  and $\Diamond\phi$ to abbreviate
$(\phi\supset\psi)\land(\psi\supset\phi)$ and $\neg\nec\neg\phi$, 
respectively.  
For each $n\in\omega$, 
$\Box^{n}$ and $\Diamond^{n}$ denote $n$-times applications of $\Box$
and $\Diamond$, respectively.

\begin{definition}
An {\em inference rule} is a pair 
$(\gm,\phi)$ of (possibly infinite) set $\gm$ of formulas and a 
single formula $\phi$. 
A {\em proof system} is a collection of inference rules. 
Let $P$ be a proof system. If $(\emptyset,\phi)\in P$, 
we call $\phi$ an {\em axiom} of $P$. 
A formula $\phi$ is said to be {\em derivable} in $P$
if either 
$\phi$ is an axiom of $P$, 
or
there exists an inference rule 
$(\gm,\phi)$ of $P$ such that 
every $\gamma\in\gm$ is derivable in $P$. 
\end{definition}

As usual, we write an inference rule $(\gm,\phi)$ by $\dfrac{\gm}{\phi}$.

\begin{definition}\label{def-kripke-model}
A {\em Kripke frame} is a pair $\langle W,R\rangle$, 
where $W$ is a non-empty set and $R$ is a binary relation on $W$.  
A pair $\langle F,v\rangle$ is said to be 
a {\em Kripke model}, 
if $F=\langle W,R\rangle$ is a Kripke frame and 
$v$ is a mapping from $\propvar$ to $\power(W)$, 
which is called a {\em valuation} on $F$. 
For each valuation $v$ on $F$, 
the domain $\propvar$ is extended to $\Phi$ in the following way: 
\begin{enumerate}
\item
$v(\top)=W$, 
$v(\bot)=\emptyset$; 
\item 
$v(\phi\land\psi)=v(\phi)\cap v(\psi)$;
\item 
$v(\neg\phi)=
W\setminus v(\phi)$;
\item 
$v(\mdl\phi)=\mdl_{F} v(\phi)$, 
where $\Box_{F}$ is a unary operator on $\power(W)$ defined by 
$$
\mdl_{F} S
=
\{w\in W\mid \forall w'\in W((w,w')\in R\thn w'\in S)\}
$$
for any $S\subseteq W$. 
\end{enumerate}
\end{definition}

\begin{definition}\label{defvl}
Let $F$ be a Kripke frame. 
We say a formula $\phi$ is {\em valid} in $F$
($F\vl\phi$, in symbol), 
if $v(\phi)=W$ for any valuation 
$v$ on $F$. 
Let $\gm$ be a set of formulas. 
We say that $\gm$ is {\em valid} in $F$  
($F\vl\gm$, in symbol), 
if $F\vl\gamma$ for every $\gamma\in\gm$.  
We write $\defframe(\gm)$ for the class of Kripke frames 
in which $\gm$ is valid. 
Let $C$ be a class of Kripke frames. 
A formula $\phi$ is said to be {\em valid} in $C$
if $F\vl\phi$ for every $F\in C$. 
We write $\deflogic(C)$ for the 
set of formulas which are valid in $C$. 
\end{definition}

\begin{definition}\label{defrulevl}
Let $F=\langle W,R\rangle$ be a Kripke frame and  
$(\gm,\phi)$ be an inference rule. 
We say that $(\gm,\phi)$ is {\em weakly valid} in $F$, 
if 
$
F\vl\gm
$
then $F\vl\phi$. 
We say that $(\gm,\phi)$ is {\em strongly valid} in $F$, 
if for any valuation $v$ on $F$
$$
\bigcap_{\gamma\in\gm}v(\gamma)\subseteq v(\phi). 
$$
Let $C$ be a class of Kripke frames. 
We say that an inference rule is {\em weakly (or strongly) valid} in $C$ 
if it is weakly (or strongly) valid in every frame $F\in C$, respectively. 
\end{definition}

For example, the necessitation rule is weakly valid in the class of 
all Kripke frames, 
but is not strongly valid in it.

\begin{proposition}\label{wkvlsound}
Let $P$ be a proof system 
and $F$ be a Kripke frame. 
If every inference rule in $P$ is weakly valid in $F$, then 
every formula which is derivable in $P$ is valid in $F$. 
\end{proposition}

\begin{proof}
Induction on the height of the derivations. 
\end{proof}

However, the converse of Proposition \ref{wkvlsound}
does not hold, in general. We give a counterexample in Theorem \ref{wknotcw}.

\bigskip

\begin{definition}\label{defframelf}
A Kripke frame $F=\langle W,R\rangle$ is said to be {\em conversely well-founded}, 
if there exists no infinite list $(w_{i})_{i\in\omega}$ such that 
$w_{i}\in W$ and $(w_{i},w_{i+1})\in R$ for every $i\in\omega$. 
Let $F=\langle W,R\rangle$ be a Kripke frame and 
$w=w_{0}\in W$. 
We say that the {\em height from $w$} is finite, 
if the supremum of the length of lists 
$w_{0},w_{1},\ldots,w_{n}\in W$ such that 
$(w_{i},w_{i+1})\in R$ is finite. 
A Kripke frame
$F=\langle W,R\rangle$ is said to be of {\em locally finite height}, 
if for any $w\in W$, the height from $w$ is finite. 
We write $\framecw$, $\framelf$, and $\framefi$ for 
the classes of transitive Kripke frames which are
conversely well-founded, of locally finite height, 
and finite and irreflexive, respectively. 
\end{definition}

\begin{definition}
An algebra 
$\langle A;\lor,\land,-,\mdl,0,1\rangle$
is called a {\em modal algebra} if it satisfies the following conditions:
\begin{enumerate}
\item
$\langle A;\lor,\land,-,0,1\rangle$ is a Boolean algebra;
\item
$
\mdl 1=1
$
and for any $x$, $y\in A$, 
$$
\mdl x\land \mdl y=\mdl(x\land y). 
$$
\end{enumerate}

\noindent
Let $A$ and $B$ be modal algebras. 
A  map $f:A\yy B$ is called a {\em homomorphism 
of modal algebras} if it is 
a homomorphism of Boolean algebras and 
satisfies 
$
f(\mdl x)=\mdl f(x)
$
for any $x\in A$. 
An injective homomorphism is called an {\em embedding}. 
We write $\Diamond x$ for ${-}\mdl{-}x$, for each $x\in A$. 
\end{definition}

\begin{definition}
An {\em algebraic model} for modal logic is a pair $\langle A,v\rangle$, 
where $A$ is a modal algebra and 
$v$ is a mapping from $\propvar$ to $A$. 
For each valuation $v$ from $\propvar$ to $A$, 
the domain $\propvar$ is extended to $\Phi$ in the following way: 
\begin{enumerate}
\item
$v(\top)=1$, 
$v(\bot)=0$; 
\item 
$v(\phi\land\psi)=v(\phi)\land v(\psi)$;
\item 
$v(\neg\phi)=
-v(\phi)$;
\item 
$v(\mdl\phi)=\mdl v(\phi)$. 
\end{enumerate}
\end{definition}

For each formula $\phi$ and each modal algebra $A$, 
we write $A\vl\phi$ if $v(\phi)=1$ for every valuation $v$ on $A$. 
Other relations between (classes of) algebraic models 
and (sets of) formulas are defined 
in the same manner as 
Definition \ref{defvl}. 
For each set $\gm$ of formulas, 
we write $\defalg(\gm)$ for the set of modal algebras in which $\gm$ is valid.

\section{An extension of the J\'{o}nsson-Tarski representation theorem}\label{modal_algebras}

In this section, we recall the relationship between Kripke frames and modal algebras, 
and present an extension of the J\'{o}nsson-Tarski representation theorem.

\begin{definition}
For each Kripke frame $F=\langle W,R\rangle$, 
we write $\dalg(F)$ for the modal algebra
$$
\dalg(F)=
\langle
\power(W);\cup,\cap,W\setminus-,\Box_{F},\emptyset,W
\rangle,  
$$
where $\Box_{F}$ is the operator defined in Definition \ref{def-kripke-model}. 
\end{definition}

It is easy to see that for any Kripke frame 
$F=\langle W,R\rangle$ and 
any $S\subseteq\power(W)$, 
\begin{equation*}\label{framebarcan}
\mdl_{F}\bigcap S=\bigcap_{s\in S}\mdl_{F} s
\end{equation*}
holds in $\dalg(F)$.

\begin{theorem}\label{path-meet}
Let $F=\langle W,R\rangle$ be a Kripke frame. Then, the following two conditions are equivalent: 
\begin{enumerate}
\item
$F$ is a frame of locally finite height; 
\item
$\iand_{n\in\omega}{\Diamond_{F}}^{n}1=0$ holds in $\dalg(F)$. 
\end{enumerate}
\end{theorem}

\begin{proof}
For any $w\in W$, 
\begin{align*}
w\in \bigcap_{n\in\omega}{\Diamond_{F}}^{n}W
&\eq
\forall n\in\omega
\exists w_{0},w_{1},\ldots,w_{n}\in W
\left(
w=w_{0},\ (w_{i},w_{i+1})\in R
\right) \\
&\eq
\text{the height from $w$ is not finite}. 
\end{align*}
\end{proof}

\begin{theorem}\label{vlduality}
{\rm (e.g. \cite{chg-zkh97,blc-rjk-vnm01})}.  
Let $F=\langle W, R\rangle$ be a Kripke frame. 
For every formula $\phi\in\Phi$,  
$F\vl\phi$ if and only if $\dalg(F)\vl\phi$. 
\end{theorem}

Let $A$ be a modal algebra. An upward closed subset $F$ of $A$ is called a {\em filter} 
of $A$, if $x\land y\in F$ for every $x$ and $y$ in $F$. 
A filter $F$ of $A$ is said to be {\em prime}, if $F\not=A$
and $x\lor y\in F$ implies $x\in F$ or $y\in F$ for every $x$ and $y$ in $A$. 
It is easy to see that if $F$ is a prime filter of $A$, 
then either $x\in F$ or $-x\in F$ but not both, for every $x\in A$. 
Ideals of $A$ are the dual objects of filters. 
The following is an infinitary extension of prime filters.

\begin{definition}\label{def:q-filter}
{\rm
(Rasiowa-Sikorski, \cite{rsw-skr63}). 
}
Let $A$ be a modal algebra and 
$Q\subseteq\power(A)$. 
A prime filter $\alpha$ of $A$ is called 
a {\em $Q$-filter} of $A$, if for any  
$X\in Q$
$$
X\subseteq \alpha
\text{ and }
\iand X\in A   
\thn \iand X\in \alpha
$$ 
holds.  The set of all $Q$-filters of 
$A$ is denoted by $\qflt(A)$. 
\end{definition}

\begin{definition}
Let $A$ be a modal algebra and 
$Q\subseteq\power(A)$. 
We write $\dframe(A)$ for the Kripke frame 
$\langle\qflt(A),R_{Q}\rangle$, 
where $R_{Q}$ is a 
binary relation  on $\qflt(A)$ which is defined by 
$$
(\alpha,\beta)\in R_{Q}\eq \Box^{-1}\alpha\subseteq\beta,  
$$
for any  $\alpha$ and $\beta\in\qflt(A)$.  
Here, $\Box^{-1}\alpha$ denotes the inverse image of the 
unary operator $\Box$, that is, 
$\Box^{-1}\alpha=\{x\mid\Box x\in\alpha\}$.  
\end{definition}

It is easy to see that if $\alpha$ is a filter of $A$, 
then so is $\Box^{-1}\alpha$.

\begin{theorem}\label{path-meet-alg}
Let $A$ be a modal algebra which satisfies
$
\iand_{n\in\omega}\Diamond^{n}1=0
$
and $Q=\{\{\Diamond^{n}1\mid n\in\omega\}\}$. 
Then, 
$\dframe(A)$ is a frame of locally finite height. 
\end{theorem}

\begin{proof}
Suppose $\iand_{n\in\omega}\Diamond^{n}1=0$ holds in $A$. 
Take any $\alpha\in\qflt(A)$. By definition, there exists $n\in\omega$ such that 
$\Diamond^{n}1\not\in \alpha$. 
Since $\alpha$ is a prime filter of $A$, ${-}\Diamond^{n}1=\mdl^{n}0\in \alpha$. Therefore, 
the height from $\alpha$ 
is at most $n$.  Hence, $\dframe(A)$ is a frame of locally finite height. 
\end{proof}

We show that for each modal algebra 
$A$ which satisfies $\iand_{n\in\omega}\Diamond^{n}1=0$, 
there exists an embedding $\eta_{A}:A\yy\dalg(\dframe(A))$ such that 
$\iand_{n\in\omega}\eta_{A}\left(\Diamond^{n}1\right)=0$, 
where 
$Q=\{\{\Diamond^{n}1\mid n\in\omega\}\}$. 
We recall the following two theorems, which we use to show the 
extension of the J\'{o}nsson-Tarski representation theorem.

\begin{theorem}\label{th:pft}
{\rm (Prime filter theorem, e.g., \cite{dvy-prs90})}. 
Let $A$ be a Boolean algebra. 
Suppose $\alpha$ is a filter of $A$ and 
$\beta$ is an ideal of $A$ such that $\alpha\cap \beta=\emptyset$. 
Then, there exists a prime filter 
$\gamma$ of $A$ such that
$\alpha\subseteq\gamma$ and 
$\gamma\cap\beta=\emptyset$. 
\end{theorem}

\begin{theorem}\label{th:rs}
{\rm
(Rasiowa-Sikorski, \cite{rsw-skr63}). 
}
Let $A$ be a Boolean algebra and 
$Q$ be a countable subset of $\power(A)$. 
For any $a_{1}$ and  $a_{2}\in A$,  
if $a_{1}\not\leq a_{2}$ then there exists $\alpha\in\qflt(A)$ such that 
$a_{1}\in\alpha$ and $a_{2}\not\in\alpha$. 
\end{theorem}

Now, we present the extension of 
the J\'{o}nsson-Tarski representation theorem.  

\begin{theorem} \label{infrep}
Let $A$ be a modal algebra which satisfies 
\begin{equation}\label{infmeet}
\iand_{n\in\omega}\Diamond^{n}1=0. 
\end{equation}
Let $Q=\{\{\Diamond^{n}1\mid n\in\omega\}\}$. 
Define a map $\eta_{A}:A\yy\dalg\left(\dframe(A)\right)$ by 
$$x\mapsto\{\alpha\in\qflt(A)|x\in\alpha\}$$
for any $x\in A$. Then, 
$\eta_{A}$ is an embedding of modal algebras such that
\begin{equation}\label{canonicalbl}
\bigcap_{n\in\omega}\eta_{A}\left(\Diamond^{n}1\right)=\emptyset. 
\end{equation}
\end{theorem}

\begin{proof}
It is easy to check that $\eta_{A}$ is a homomorphism of Boolean algebras. 
By Theorem \ref{th:rs}, $\eta_{A}$ is injective.  
The equation (\ref{canonicalbl}) follows from the definition of  $Q$-filters
and (\ref{infmeet}), as follows: 
\begin{align*}
\alpha\in\bigcap_{n\in\omega}\eta_{A}\left(\Diamond^{n}1\right)
&\eq
\forall n\in\omega
\left(
\Diamond^{n}1\in \alpha
\right)\\
&\eq
0\in \alpha. 
\end{align*}

\noindent
We show $\eta_{A}(\Box x)=\Box_{\dframe(A)}\eta_{A}(x)$. 
Suppose that $\alpha\in\eta_{A}(\Box x)$ and $(\alpha,\beta)\in R_{A}$. 
Then, $x\in \beta$ by definition of $R_{A}$, hence 
$\beta\in\eta_{A}(x)$. 
Since $\beta$ is taken arbitrarily, 
$\alpha\in\Box_{\dframe(A)}\eta_{A}(x)$.  
Conversely, suppose 
$\alpha\not\in\eta_{A}(\Box x)$. 
Since $\alpha$ is a $Q$-filter, there exists $k\geq 1$ such that 
$\Diamond^{k}1\not\in\alpha$. 
Since $\alpha$ is a prime filter, 
${-}\Diamond^{k}1=\Box{-}\Diamond^{k-1}1\in\alpha$. 
By Theorem \ref{th:pft}, there exists a prime filter $\gamma$
such that $\Box^{-1}\alpha\subseteq\gamma$ and 
$x\not\in \gamma$. 
$\gamma$ is in $\qflt(A)$, since $\Diamond^{k-1}1\not\in\gamma$, because 
$-\Diamond^{k-1}1\in\gamma$. 
Therefore, $\alpha\not\in\Box_{\dframe(A)}\eta_{A}(x)$, since 
$\gamma\not\in\eta_{A}(x)$ and $(\alpha,\gamma)\in R_{Q}$. 
\end{proof}

\begin{corollary}\label{vldalg}
Let $A$ be a modal algebra and $Q=\{\{\Diamond^{n}1\mid n\in\omega\}\}$. 
For every formula $\phi\in\Phi$,  
if $\dframe(A)\vl\phi$ then $A\vl\phi$. 
\end{corollary}

\begin{proof}
Suppose $A\not\vl\phi$. Then, there exists a valuation $u$ on $A$ such that 
$u(\phi)\not=1$. Let $v$ be a valuation on $\dalg(\dframe(A))$ such that 
$v(p)=\eta_{A}(u(p))$. Since $\eta_{A}$ is injective, $v(\phi)\not=1$. 
Hence, $\dalg(\dframe(A))\not\vl\phi$. 
By Theorem \ref{vlduality}, $\dframe(A)\not\vl\phi$. 
\end{proof}

Theorem \ref{infrep} can be applied to some modal algebras 
which do not satisfy the conditions of the following 
infinitary representation theorem given in \cite{tnk-ono98}. 

\begin{theorem} {\rm (\cite{tnk-ono98})}. 
\label{ma-irt}
Let $A$ be a modal algebra and 
$Q$ a countable subset of $\power(A)$
which satisfies the following conditions:
\begin{enumerate}
\item
$\forall z\in A \forall X\in Q
\left(\{\Box(z\lor x)\mid x\in X\}\in Q\right)$;

\item
$\forall X\in Q
\left(\iand X\in A\right)$;

\item
$\forall X\in Q
\left(\iand\Box X=\Box\iand X\right)$.

\end{enumerate}

\noindent
Then, 
a map
$\eta:A\yy\dalg\left(\dframe(A)\right)$ defined by 
$\eta:x\mapsto\{\alpha\in\qflt(A)\mid  x\in\alpha\}$
is an embedding of modal algebras 
which satisfies 
$\eta\left(\iand X\right)=\bigcap\eta[X]$
for every $X\in Q$. 
\end{theorem}

For countable modal algebras which satisfy
$\iand_{n\in\omega}\Diamond^{n}1=0$,  
we can show the existence of the embedding 
which preserves the infinite meet 
by Theorem \ref{ma-irt} and the following 
Lemma.

\begin{lemma}\label{localbarcan}
Let $A$ be a modal algebra such that 
$\iand_{n\in\omega}\Diamond^{n}1=0$. 
Then 
for any natural number $k\in\omega$ and any $x_{1},\ldots,x_{k}\in A$,
\begin{multline}\label{infdist}
\iand_{n\in\omega}
\Box\left(
x_{k}\lor
\Box(
x_{k-1}\lor
\cdots
\Box(
x_{2}
\lor
\Box(
x_{1}
\lor
\Diamond^{n}1
)
)
\cdots
)
\right)\\
=
\Box\left(
x_{k}\lor
\Box(
x_{k-1}\lor
\cdots
\Box
(
x_{2}
\lor
\Box
x_{1}
)
\cdots
)
\right).
\end{multline}

\noindent
Especially,
\begin{equation*}
\iand_{n\in\omega}\Box^{k}\Diamond^{n}1=\Box^{k} 0. 
\end{equation*}
\end{lemma}

\begin{proof}
Take any $k\in\omega$. 
It is clear that
the right hand side of (\ref{infdist}) is a lower bound of the set of 
elements in the infinite meet of the left hand side. 
Suppose that there exists $y\in A$ which satisfies
\begin{equation}\label{yleq}
y\leq 
\Box\left(
x_{k}\lor
\Box(
x_{k-1}\lor
\cdots
\Box(
x_{2}
\lor
\Box(
x_{1}
\lor
\Diamond^{n}1
)
)
\cdots
)
\right)
\end{equation}
for any $n\in\omega$ and 
\begin{equation}\label{ynotleq}
y\not\leq 
\Box\left(
x_{k}\lor
\Box(
x_{k-1}\lor
\cdots
\Box(
x_{2}
\lor
\Box x_{1}
)
\cdots
)
\right). 
\end{equation}
Let 
$$
Q=\{\{\Diamond^{n}1\mid n\in\omega\}\}. 
$$
By Theorem \ref{th:rs}, 
there exists a $Q$-filter $\alpha$ of 
$A$ such that 
$y\in \alpha$
and 
\begin{equation*}
\Box\left(
x_{k}\lor
\Box(
x_{k-1}\lor
\cdots
\Box(
x_{2}
\lor
\Box x_{1}
)
\cdots
)
\right)
\not\in \alpha.
\end{equation*}
By $\iand_{n\in\omega}\Diamond^{n}1=0$, 
there exists $m\in\omega$ such that 
$\Diamond^{m}1\not\in\alpha$. 
Since $\Diamond 1\leq 1$ and the operator $\Diamond$ is order preserving, 
$\Diamond^{n+1} 1\leq \Diamond^{n} 1$ for any $n\in\omega$. 
Hence, 
$\Diamond^{m+k+1}1\not\in\alpha$. 
Then, 
$-\Diamond^{m+k+1}1=\Box^{k+1}({-}\Diamond^{m}1)\in\alpha$. 
By 
\begin{align*}
\Box^{k+1}({-}\Diamond^{m}1)
&\land
\Box\left(
x_{k+1}\lor
\Box(
x_{k}\lor
\cdots
\Box(
x_{2}
\lor
\Box(
x_{1}
\lor
\Diamond^{m}1
)
)
\cdots
)
\right)\\
&\leq
\Box\left(
x_{k+1}\lor
\Box(
x_{k}\lor
\cdots
\Box(
x_{2}
\lor
\Box(
x_{1}
\lor
({-}\Diamond^{m}1\land\Diamond^{m}1)
)
)
\cdots
)
\right)\\
&=
\Box\left(
x_{k+1}\lor
\Box(
x_{k}\lor
\cdots
\Box(
x_{2}
\lor
\Box
x_{1}
)
\cdots
)
\right)
\end{align*}
and (\ref{yleq}), we have 
$$
\Box\left(
x_{k+1}\lor
\Box(
x_{k}\lor
\cdots
\Box(
x_{2}
\lor
\Box
x_{1}
)
\cdots
)
\right)
\in\alpha, 
$$
which is contradiction. 
\end{proof}

Suppose $A$ is a countable modal algebra which satisfies 
$\iand_{n\in\omega}\Diamond^{n}1=0$. Define $Q\subseteq\power(A)$ as follows: 
\begin{align}
&Q_{0}
=
\left\{
\{\Diamond^{n}1\mid n\in\omega\}
\right\}
;\notag\\
&Q_{n+1}
=
\left\{
\left\{\Box(z\lor x)\mid x\in X\right\}
\mid
z\in A, X\in Q_{n}
\right\};\notag\\
&
Q=\bigcup_{n\in\omega}Q_{n}. \label{setq}
\end{align}
Then, $Q$ is countable. 
Hence, 
Lemma \ref{localbarcan} and 
Theorem \ref{ma-irt} imply that 
there exists an embedding $\eta_{A}:A\yy\dalg(\dframe(A))$ 
which satisfies 
$\eta_{A}\left(\iand_{n\in\omega}\Diamond^{n}1\right)=0$. 
However, 
there exists a modal algebra which satisfies
$\iand_{n\in\omega}\Diamond^{n}1=0$ but the cardinality of $Q$ in 
(\ref{setq}) is uncountable, as follows. 
Let $F=\langle W,R\rangle$ be a Kripke frame where
$W$ is the set of mapping from $\omega$ to $\omega$ and 
$R$ is a binary relation on $W$ such that 
$$
(f,g)\in R\eq
\forall i\in\omega \left(g(i)< f(i)\right),  
$$
for each $f$ and $g$ in $W$.  
Then, $1_{\dalg(F)}=W$ and $0_{\dalg(F)}=\emptyset$, and 
$\dalg(F)$ satisfies $\iand_{n\in\omega}\Diamond_{F}^{n}1_{\dalg(F)}=0_{\dalg(F)}$. 
We show that the cardinality of $Q$ in 
(\ref{setq}) is uncountable. 
For each mapping $f:\omega\yy\omega$, let 
$$
\da f
=
\{g\mid
g:\omega\yy\omega,\ \forall i\in\omega(g(i)\leq f(i)) 
\}. 
$$
Then, 
$$
\{
\mdl_{F}
\left(
\da f\cup\Diamond_{F}^{n}W
\right)
\mid
n\in\omega
\}\in Q_{1},  
$$
and 
$$
\iand\{
\mdl_{F}
\left(
\da f\cup\Diamond_{F}^{n}W
\right)
\mid
n\in\omega
\}
=
\da(f+1). 
$$
Therefore, 
$$
\aleph_{0}<\sharp Q_{1}\leq\sharp Q. 
$$
Thus, $\dalg(F)$ does not satisfy the conditions of 
Theorem \ref{ma-irt}. On the other hand,  
Theorem \ref{infrep} can be applied to $\dalg(F)$.

\section{An $\omega$-rule for $\propgl$}\label{non-compact_rule}

In this section, we introduce a proof system $\pglsys$, which has 
an $\omega$-rule, and show that $\pglsys$ is a proof system for 
the logic $\propgl$ of provability. 
The logic $\propgl$ is the smallest normal modal logic 
which includes $\logick$ and the L\"{o}b formula 
$
\Box(\Box p\supset p)\supset \Box p
$.
It is known that 
$\defframe(\propgl)=\framecw$ and 
$\deflogic(\framecw)=\propgl$. 
It is also known that 
$\deflogic(\framefi)=\propgl$. 
As 
$\framefi\subseteq\framelf\subseteq\framecw$, 
$\propgl$ is sound and complete with respect to $\framelf$.

The axioms of $\pglsys$ are all classical tautologies and 
the following axiom schemata of modal logic: 
\begin{align*}
\text{($\mathsf{K}$)}
&:
\Box(p\supset q)\supset(\Box p\supset \Box q);\\
\text{({\sf 4})}
&:
\Box p\supset\Box\Box p.
\end{align*}
The inference rules of $\pglsys$ are 
modus ponens, uniform substitution, generalization and 
the following $\omega$-rule: 

\medskip

\noindent
$$
\dmndn: \ 
\frac
{\phi\supset\Diamond^{n}\top\hspace{10pt}
(\text{$\forall n\in\omega$})}
{\phi\supset\bot}.
$$

The rule $\dmndn$ has the following property.

\begin{theorem}\label{lfstvl}
For any transitive Kripke frame $F$, 
$F\in\framelf$ if and only if $\dmndn$ is strongly valid in $F$. 
\end{theorem}

\begin{proof}
Suppose $F\in\framelf$. Take any valuation $v$ on $F$ and any formula $\phi$. 
By Theorem \ref{path-meet}, 
\begin{equation*}
\bigcap_{n\in\omega}v(\phi\supset\Diamond^{n}\top)
=
v(\phi)\supset
\bigcap_{n\in\omega}\Diamond_{F}^{n}W
=
v(\phi)\supset
\emptyset
=
v(\phi\supset\bot). 
\end{equation*}
Conversely, suppose $\dmndn$ is strongly valid in a transitive Kripke 
frame $F=\langle W,R\rangle$. 
Let $\phi$ in $\dmndn$ be $\top$. Then, for any valuation $v$ on $F$, 
$$
\bigcap_{n\in\omega}\Diamond^{n}W
=
\bigcap_{n\in\omega}v\left(\top\supset\Diamond^{n}\top\right)
\subseteq
v\left(\top\supset\bot\right)
=
\emptyset.
$$
Hence, $F\in\framelf$, by 
Theorem \ref{path-meet}. 
\end{proof}

It is shown in \cite{tnk18} that $\pglsys$ is a proof system 
for $\propgl$, that is, 
\begin{equation}\label{propglpglsys}
\phi\in\propgl\text{ if and only if }\gldb\phi
\end{equation}
for any formula $\phi$. 
In this paper, we give another proof of the only if part of this fact. 
To make the article self-contained, we first show the if part.

\begin{theorem}\label{glsoundness}
{\rm (\cite{tnk18}).}
For any formula $\phi$, 
if 
$\gldb\phi$, then 
$\phi\in\propgl$. 
\end{theorem}

\begin{proof}
Since $\propgl=\deflogic(\framelf)$, 
it is enough to show that $\pglsys$ is sound with respect to 
$\framelf$. 
By Theorem \ref{lfstvl}, 
$\dmndn$ is weakly valid in $\framelf$. 
It is clear that all of other rules of $\pglsys$ 
are weakly valid in $\framelf$. 
\end{proof}

In \cite{tnk18}, the only if part of (\ref{propglpglsys}) is 
proved by showing Kripke completeness of $\pglsys$ with respect to $\framelf$. 
In this paper, we give a direct syntactical proof of it.

\begin{theorem}
{\rm (\cite{tnk18}).}
For any formula $\phi$, 
if 
$\phi\in\propgl$, then 
$\gldb\phi$. 
\end{theorem}

\begin{proof}
It is enough to show that 
the L\"{o}b formula is derivable in $\pglsys$.  
We show that 
\begin{equation}\label{upperseq}
\gldb\neg\left(\Box(\Box p\supset p)\supset \Box p\right)\supset\Diamond^{n}\top
\end{equation}
for any $n\in\omega$, by  
induction on $n\in\omega$. 
The case $n=0$ is trivial. 
Suppose (\ref{upperseq}) holds for $n$. 
Then by applying classical derivations, 
$$
\gldb
\left(\neg\Diamond^{n}\top\land\Box(\Box p\supset p)
\land\left(\Box p\supset p\right)\right)\supset  p. 
$$
Therefore, by necessitation, ({\sf K}), and classical derivations, 
$$
\gldb
\left(\Box\neg\Diamond^{n}\top\land\Box\Box(\Box p\supset p)
\land\Box\left(\Box p\supset p\right)\right)\supset  \Box p.
$$
By the axiom ({\sf 4})  and classical derivations, 
$$
\gldb
\left(\Box\neg\Diamond^{n}\top
\land\Box\left(\Box p\supset p\right)\right)\supset  \Box p.
$$
Therefore, 
$$
\gldb\neg\left(\Box(\Box p\supset p)\supset \Box p\right)\supset\Diamond^{n+1}\top. 
$$
Hence, (\ref{upperseq}) holds for any $n\in\omega$. 
By $\dmndn$, 
$\gldb \Box(\Box p\supset p)\supset \Box p$. 
\end{proof}

In \cite{tnk18}, 
the proof of Kripke completeness of $\pglsys$ with respect to $\framelf$ 
is given by Henkin construction.   
In the following, 
we give another proof of it in an algebraic manner.

\begin{theorem}\label{glcompleteness}
{\rm (\cite{tnk18}).}
For any formula $\phi$, 
if 
$\framelf\vl\phi$, then 
$\gldb\phi$. 
\end{theorem}

\begin{proof}
Define a binary relation $\sim$ on the set $\Phi$ of all formulas by 
$$
\phi_{1}\sim\phi_{2}\eq
\gldb\phi_{1}\equiv\phi_{2}.
$$
Let $A$ be the quotient modal algebra of the set $\Phi$ of all formulas
modulo $\sim$. For each formula $\phi$, we write $[\phi]$ for the equivalence class 
of $\phi$. 
Then, by $\dmndn$,  
\begin{equation}\label{lainfmeet}
\iand_{n\in\omega}\Diamond^{n}[\top]=[\bot]
\end{equation}
holds in $A$. 
Define $Q\subseteq\power(A)$ by 
$Q=\{\{\Diamond^{n}[\top]\mid n\in\omega\}\}$. 
$\dframe(A)$ is transitive, since 
$\Box[\psi]\leq\Box\Box[\psi]$ holds for any formula $\psi$. 
Hence, 
$\dframe(A)\in\framelf$,  
by Theorem \ref{path-meet-alg}. 
Suppose that $\phi$ is not derivable in $\pglsys$. Then, 
$A\not\vl\phi$, by definition of $A$. 
Hence, $\dframe(A)\not\vl\phi$ by Theorem \ref{vldalg}. 
Therefore, $\framelf\not\vl\phi$. 
\end{proof}

It is known that $\propgl$ is not canonical. 
That is, $\deflogic(F)\subsetneqq\propgl$, where $F=\langle W,R\rangle$ is the 
canonical frame of $\propgl$. 
So, in the well-known proof of Kripke completeness of $\propgl$, 
the binary relation $R$ on the canonical frame $F$ is replaced by the following 
$R'$: 
$(\alpha,\beta)\in R' $ if and only if
(1) for all $\Box\phi\in \alpha$, both $\Box\phi\in\beta$ and $\phi\in\beta$, 
and (2) there exists some $\Box\phi\in\beta$ 
such that $\Box\phi\not\in\alpha$.
On the other hand, 
the binary relation on $\dframe(A)$ given in the proof of 
Theorem \ref{glcompleteness}
is the restriction of the 
relation $R$ of the canonical model to $\qflt(A)$. 
Therefore, 
$\dframe(A)$ is a subframe of the canonical frame of $\propgl$.   

\medskip

\begin{corollary}\label{algcompleteness}
Let $C$ be a class of modal algebras which satisfies 
(\ref{infmeet}) and $\mdl x\leq\mdl\mdl x$ for any $x\in A$. 
Then, 
$\phi\in\propgl$ if and only if  
$C\vl\phi$, 
for any formula $\phi$.  
\end{corollary}

\begin{proof}
First, suppose $\phi\in\propgl$. Then, $\framelf\vl\phi$. 
Let $A$ be a modal algebra which satisfies 
(\ref{infmeet}) and $\mdl x\leq\mdl\mdl x$ for any $x\in A$, 
and let 
$Q=\{\{\Diamond^{n}[\top]\mid n\in\omega\}\}$. 
Then, $\dframe(A)\vl\phi$, by Theorem \ref{path-meet-alg}. 
Therefore, $A\vl\phi$, by Corollary \ref{vldalg}. 
Next, suppose $\phi\not\in\propgl$. 
Then, $\pglsys\not\db\phi$ by Theorem \ref{glsoundness}. 
Let $A$ be the quotient modal algebra 
given in the proof of Theorem \ref{glcompleteness}. 
Then, $A\not\vl\phi$. 
Hence, $C\not\vl\phi$,   
as $A\in C$. 
\end{proof}

\section{Classes of Kripke models for $\propgl$}\label{modelsofpglsys}

In this section, we discuss relationship among some classes of 
Kripke frames, each of which characterizes $\propgl$. 
Let 
$\frmwk$ and $\frmst$ be classes of Kripke frames such that  
\begin{align*}
\frmwk
&=
\left\{F\mid
\text{$F$ is transitive and $\dmndn$ is weakly valid in $F$}
\right\};\\
\frmst
&=
\left\{F\mid
\text{$F$ is transitive and $\dmndn$ is strongly valid in $F$}
\right\}.
\end{align*}
We have already discussed that 
$$
\deflogic({\framefi})
=
\deflogic({\framelf})
=
\deflogic({\frmst})
=
\deflogic({\frmwk})
=
\deflogic({\framecw})
=
\propgl,  
$$
that is,  
all of $\framefi$, $\framelf$, $\frmst$, $\frmwk$, and $\framecw$ 
characterize $\propgl$. 
It is proved in Theorem \ref{lfstvl} that 
$\framelf=\frmst$.  
In the rest of this paper, 
we show that 
$$
\framefi
\subsetneqq
\framelf
=
\frmst
\subsetneqq
\frmwk
\subsetneqq
\framecw.   
$$

\begin{theorem}\label{lfnotwvl}
$\framelf\subsetneqq\frmwk$. 
\end{theorem}

\begin{proof}
Since $\framelf=\frmst$, 
$\framelf\subseteq\frmwk$. 
We show  
$\framelf\not=\frmwk$. 
Take a Kripke frame  $F=\langle\omega+1,>\rangle$ (see Figure \ref{figure}). 
Then, 
$F\not\in\framelf$, since the supremum of the length of the paths from $\omega$ is infinite. 
We show that $F\in\frmwk$. 
Suppose that $\dmndn$ is not weakly valid in $F$. 
Then, there exists a formula $\phi$ 
such that
\begin{equation}\label{upper}
\forall n\in\omega
\forall
v:\propvar\yy\power(\omega+1)
\left(
v\left(\phi\right)
\subseteq
{\Diamond_{F}}^{n}(\omega+1)
\right), 
\end{equation}
and there exists $u:\propvar\yy\power(\omega+1)$ such that 
\begin{equation}\label{lower}
\emptyset\not= u\left(\phi\right). 
\end{equation}
By (\ref{upper}), 
\begin{equation}\label{lessthanomega}
v\left(\phi\right)\subseteq\{\omega\} 
\end{equation}
for any $v:\propvar\yy\power(\omega+1)$, and 
by (\ref{lower}) and (\ref{lessthanomega}), 
\begin{equation}\label{equaltoomega}
u\left(\phi\right)=\{\omega\}.  
\end{equation}
Now, for each $n\in\omega$ and each 
$v:\propvar\yy\power(\omega+1)$, we define a map 
$v_{n}:\propvar\yy\power(\omega+1)$ as follows: 
for any $p\in\propvar$, 
$$
v_{n}(p)
=
\begin{cases}
v(p)\cup\{n\} & (\omega\in v(p))\\
v(p)\setminus\{n\} & (\omega\not\in v(p))
\end{cases}.
$$
Easy induction on the construction of the formulas shows that 
for any formula $\psi$ and any natural number $m<n$, 
\begin{equation}\label{lessthann}
m\in v(\psi)\eq m\in v_{n}(\psi)
\end{equation}
holds. 
Also, the following claim holds: 
\begin{claim}\label{claim}
For any formula $\psi$ and 
any $v:\propvar\yy\power(\omega+1)$, 
there exists $N\in\omega$ such that 
for any $n\geq N$ and any subformula $\rho$ of $\psi$, 
\begin{equation*}\label{valuation}
\omega\in v(\rho)\eq n\in v_{n}(\rho). 
\end{equation*}
\end{claim}

Proof of the claim: Induction on the construction of 
$\psi$: 

\begin{figure}[hbtp]
\centering
\includegraphics[width=0.4\textwidth]{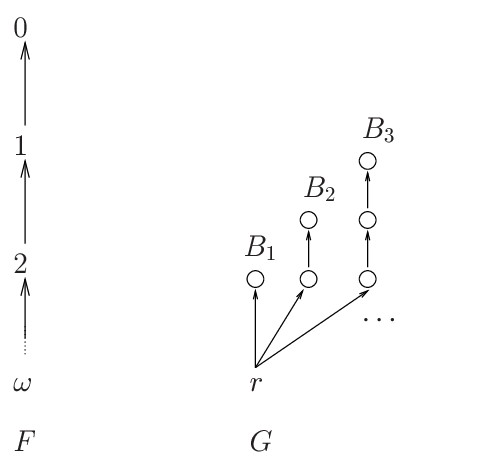}
\caption{}
\label{figure}
\end{figure}

\medskip
\noindent
$\psi=p$:
For any $n\in\omega$, 
$$
\omega\in v(p)\eq n\in v_{n}(p), 
$$
by definition of $v_{n}$. 
Therefore, the claim holds for $N=0$.

\medskip
\noindent
$\psi=\rho_{1}\land\rho_{2}$: 
By the induction hypothesis, for each $i=1$ or $2$, there exist 
$N_{i}\in\omega$ such that the claim holds for any $n\geq N_{i}$ 
and any subformula of $\rho_{i}$, respectively. 
Let $N=\max\{N_{1},N_{2}\}$. Then, for any $n>N$, 
\begin{align*}
\omega\in v(\rho_{1}\land\rho_{2})
&\eq
\text{
$\omega\in v(\rho_{1})$ and $\omega\in v(\rho_{2})$
}\\
&\eq
\text{
$n\in v_{n}(\rho_{1})$ and $n\in v_{n}(\rho_{2})$
}\\
&\eq
n\in v_{n}(\rho_{1}\land\rho_{2}). 
\end{align*}

\noindent
$\psi=\neg\rho$:
Take the same $N\in\omega$ for $\rho$. Then, for any $n\geq N$,  
\begin{align*}
\omega\in v(\neg\rho)
\eq
\omega\not\in v(\rho)
\eq
n\not\in v_{n}(\rho)
\eq
n\in v_{n}(\neg\rho). 
\end{align*}

\medskip
\noindent
$\psi=\Box\rho$:
By the induction hypothesis, there exist 
$N\in\omega$ such that the claim holds for any $n\geq N$ 
and any subformula of $\rho$. 
First, 
suppose that 
$\omega\in v(\Box\rho)$. 
Then, $k\in v(\rho)$, for any $k\in\omega$. 
Hence, for any $n\in\omega$ and any $m<n$, 
$m\in v_{n}(\rho)$
by (\ref{lessthann}). 
Therefore, $n\in v_{n}(\Box\rho)$ for any $n\in\omega$. 
Hence, the claim holds for $N$. 
Next, suppose that $\omega\not\in v(\Box\rho)$. 
Then, there exists $k\in\omega$ such that $k\not\in v(\rho)$. 
If $n>k$, 
$k\not\in v_{n}(\rho)$ by (\ref{lessthann}), and therefore, 
$n\not\in v_{n}(\Box\rho)$.
Hence, the claim holds for $\max\{N,k+1\}$. 
This complete the proof of the claim.

By (\ref{equaltoomega}) and Claim \ref{claim}, there exists $N\in\omega$ 
such that 
$$
N\in u_{N}\left(\phi\right). 
$$
This contradict to (\ref{lessthanomega}). 
Hence, $\dmndn$ is weakly valid in $F$. 
\end{proof}

\begin{corollary}
The class $C$ of modal algebras which is defined by 
$\iand_{n\in\omega}\Diamond^{n}1=0$ and $\mdl x\leq\mdl\mdl x$ 
is not a variety. 
\end{corollary}

\begin{proof}
It is easy to see that we can identify the equations of the language 
of modal algebras with modal formulas.  
Therefore, it is enough to show that $C\subsetneqq \defalg(\deflogic(C))$ to prove 
the corollary. 
By Corollary \ref{algcompleteness}, 
$\deflogic(C)=\propgl$. 
Hence, 
$\defalg(\deflogic(C))$ is the class of all modal algebras in which $\propgl$ is valid. 
Let $F$ be the Kripke frame given in the proof of Theorem \ref{lfnotwvl}. 
Then, $\dalg(F)\not\in C$ by Theorem \ref{path-meet}. 
On the other hand, $\propgl$ is valid in $\dalg(F)$ by Theorem \ref{vlduality}. 
Therefore, 
$\dalg(F)\in\defalg(\deflogic(C))\setminus C$. 
\end{proof}

\begin{theorem}\label{wknotcw}
$\frmwk\subsetneqq\framecw$.
\end{theorem}

\begin{proof}
Suppose $F\in\frmwk$. Then, $F\vl\propgl$, 
since every inference rule in $\pglsys$ is weakly valid in $F$. 
Hence, $\frmwk\subseteq\framecw$. 
We show that $\frmwk\not=\framecw$.
Let $G=\langle W,R\rangle$ be a Kripke frame which consists of the 
root $r$ and disjoint branches $B_{n}$ for each $n\in\omega$, 
where $B_{n}$ is order isomorphic to $\langle n,<\rangle$ for each $n\in\omega$
(Figure \ref{figure}). 
It is clear that 
$G\in\framecw$.  
We show that $G\not\in\frmwk$. 
Let
$$
\phi=\Box\left(p\land\Box p\supset q\right)\lor
\Box\left(q\land\Box q\supset p\right).  
$$
It is easy to prove that for any Kripke frame $\langle W,R\rangle$ and any $x\in W$, 
$\phi$ satisfies the following: 
\begin{multline*}
\forall v:\propvar\yy \power(W)
\left(
\langle W,R,v\rangle,x\vl\phi
\right)\\
\eq
\forall y\in W\forall z\in W
\left(
(x,y),\  (x,z)\in R,\ y\not=z
\thn
(y,z)\in R\text{ or }(z,y)\in R
\right). 
\end{multline*}
Then, for any $n\in\omega$, $G\vl\neg \phi\supset\Diamond^{n}\top$, 
because, for every $v:\propvar\yy \power(W)$, 
$w\in v(\phi)$ for any $w\not=r$ 
and $r\in v(\Diamond^{n}\top)$. 
However, $G\not\vl\phi$, since $r\not\in v(\phi)$. 
Therefore, $\dmndn$ for $\neg \phi$ is not weakly valid in $G$. 
\end{proof}

By Theorem \ref{wknotcw}, $\dmndn$ is not weakly valid in $\framecw$. 
Therefore, although $\pglsys$ is sound with respect to $\framecw$, 
the soundness cannot be proved by simple induction on the height of 
the derivations in $\pglsys$.

\newcommand{\noop}[1]{}


\begin{thebibliography}{10}

\bibitem{blc-rjk-vnm01}
Patrick Blackburn, Maarten de~Rijke, and Yde Venema.
\newblock {\em Modal Logic}.
\newblock Cambridge, third edition, 2001.

\bibitem{bls93}
George Boolos.
\newblock {\em The logic of provability}.
\newblock Cambridge University Press, 1993.

\bibitem{chg-zkh97}
Alexander Chagrov and Michael Zakharyaschev.
\newblock {\em Modal Logic}.
\newblock Oxford University Press, 1997.

\bibitem{dvy-prs90}
Brian~A. Davey and Hilary~A. Priestley.
\newblock {\em Introduction to Lattices and Order}.
\newblock Cambridge University Press, 1990.

\bibitem{hgh-crs96}
George~E. Hughes and Maxwell~J. Cresswell.
\newblock {\em A New Introduction to Modal Logic}.
\newblock Routledge, 1996.

\bibitem{jns-trs51}
Bjarni J{\'o}nsson and Alfred Tarski.
\newblock {Boolean} algebras with operators {I}.
\newblock {\em American Journal of Mathematics}, 73:891--931, 1951.

\bibitem{jns-trs52}
Bjarni J{\'o}nsson and Alfred Tarski.
\newblock {Boolean} algebras with operators {II}.
\newblock {\em American Journal of Mathematics}, 74:127--162, 1952.

\bibitem{rsw-skr63}
Helena Rasiowa and Roman Sikorski.
\newblock {\em The Mathematics of Metamathematics}.
\newblock PWN-Polish Scientific Publishers, 1963.

\bibitem{tnknoncpt}
Yoshihito Tanaka.
\newblock Model existence in non-compact modal logic.
\newblock {\em Studia Logica}, 67:61--73, 2001.

\bibitem{tnk18}
Yoshihito Tanaka.
\newblock A cut-free proof system for a predicate extension of the logic of
  provability.
\newblock {\em Reports on Mathematical Logic}, 53:97--109, 2018.

\bibitem{tnk-ono98}
Yoshihito Tanaka and Hiroakira Ono.
\newblock The {Rasiowa-Sikorski} lemma and {Kripke} completeness of predicate
  and infinitary modal logics.
\newblock In Michael Zakharyaschev, Krister Segerberg, Maarten de~Rijke, and
  Heinrich Wansing, editors, {\em Advances in Modal Logic}, volume~2, pages
  419--437. CSLI Publication, 2000.

\end{thebibliography}
\end{document}